\documentclass[12pt]{amsart}

 \usepackage{amscd, graphicx}
 \usepackage{tikz-cd}
 \usepackage{ifthen}    
 \usepackage{amssymb}   
 \usepackage{amstext}
 \usepackage{amsfonts}
 \usepackage{amsmath}   
 \usepackage{latexsym}  
 \usepackage{color}
 \usepackage{epsfig}
 \usepackage{pstricks, pst-node,pst-text,pst-tree}
 \usepackage[all]{xy}
 \usepackage{array}
 \usepackage{amsthm}
 \usepackage{eucal, mathrsfs} 
 \usepackage{hyperref}
 \usepackage{mathtools}

\newcommand{\numberset}[1]{\mathbb{#1}}    
\newcommand{\C}{\numberset{C}}  
 
\newcommand{\R}{\numberset{R}}  
\newcommand{\Z}{\numberset{Z}}  
\newcommand{\PP}{\numberset{P}}  

\newcommand{\T}{\numberset{T}}

\newcommand{\FF}{\numberset{F}}

\newcommand{\F}{\mathcal F}

\newcommand{\inn}[2]{ \langle {#1}, {#2} \rangle}

\theoremstyle{definition}

\newtheorem{thm}{Theorem}[section]
\newtheorem{prop}[thm]{Proposition}
\newtheorem{lem}[thm]{Lemma}
\newtheorem{cor}[thm]{Corollary}
\newtheorem{rem}[thm]{Remark}
\newtheorem{ex}[thm]{Example}
\newtheorem{defi}[thm]{Definition}

\newtheorem{con}[thm]{Conjecture}

\DeclareMathOperator{\vol}{vol}

\DeclareMathOperator{\Hom}{Hom}

\DeclareMathOperator{\Sk}{Skel}

\DeclareMathOperator{\trop}{trop}

\newboolean{mynotes}\setboolean{mynotes}{true}

\newcommand{\mycomments}[1]{
           \ifthenelse{\boolean{mynotes}}
                      {#1}{}
           }

\begin{document}

\title[{First Betti Number of real Calabi Yau Hypersurfaces: examples}]{{First Betti Number of real Calabi Yau Hypersurfaces: examples}}
\author{Diego Matessi, Arthur Renaudineau}

\begin{abstract}
Continuing the investigation of real Calabi–Yau hypersurfaces in toric varieties obtained by patchworking, we present a new theorem concerning the computation of their first Betti number using mirror symmetry. Although the proof of this result will appear elsewhere, we focus here on its consequences and applications to the topology of real Calabi–Yau hypersurfaces.
\end{abstract}

\maketitle
\tableofcontents

\section{Introduction}
\subsection{Tropical mirror symmetry and connectedness} 
In our previous work \cite{MR24}, we studied the tropical homology groups of mirror pairs of tropical Calabi–Yau hypersurfaces arising from unimodular central subdivisions of reflexive polytopes. Let $\Delta$ and $\Delta^{\circ}$ be dual $(n+1)$-dimensional reflexive polytopes, and suppose that each admits a primitive central subdivision. These subdivisions yield resolutions of the toric varieties associated with $\Delta$ and $\Delta^{\circ}$, as well as crepant resolutions of their anticanonical hypersurfaces, thereby defining a mirror pair of smooth $n$-dimensional Calabi–Yau varieties $X$ and $X^{\circ}$ (see \cite{batyr:dual_pol}). Consequently, the varieties $X$ and $X^{\circ}$ satisfy the familiar mirror symmetry relation among Hodge numbers,
\begin{equation} \label{mirror:hdg}
h^{p,q}(X) = h^{n-p,q}(X^{\circ}).
\end{equation}
In \cite{MR24}, we further demonstrated that the same combinatorial data determine tropical hypersurfaces $X_{\mathrm{trop}}$ and $X_{\mathrm{trop}}^{\circ}$ which likewise exhibit a mirror symmetry relation:
\begin{thm}[\cite{MR24}] \label{thm1MR24} Let $A$ be a commutative ring. Given central primitive subdivisions of both $\Delta$ and $\Delta^{\circ}$ with associated mirror tropical hypersurfaces $X_{\text{trop}}$ and $X^{\circ}_{\text{trop}}$, we have canonical isomorphisms 
\[ H_{q}(X_{\text{trop}}; \mathcal F_p \otimes A ) \cong H_{q}(X_{\text{trop}}^{\circ}; \mathcal F_{n-p} \otimes A). \]
\end{thm}
Here $H_{q}(X_{\trop}; \mathcal F_p )$ denotes the tropical homology groups as introduced by Itenberg, Katzarkov, Mikhalkin and Zharkov  \cite{trop_hom}. These groups are defined as the cellular homology of a certain non-constant cosheaf $\mathcal F_p$ of $\Z$-modules defined over $X_{\text{trop}}$. 
\\ In the second part of \cite{MR24}, we investigated real Calabi--Yau hypersurfaces obtained via patchworking arising from a primitive central triangulation \(T\) of a reflexive polytope \(\Delta\). In this setting, the choice of a sign distribution on the integral points of \(\Delta\) uniquely determines a divisor \(D\) in the Picard group (with \(\mathbb{F}_2\)-coefficients) of the resolved dual toric variety \(\mathbb{C}\Sigma_T\) corresponding to \(\Delta^{\circ}\). Indeed, the integral points on the boundary of \(\Delta\) correspond to rays of the fan of \(\mathbb{C}\Sigma_T\), and hence to toric divisors. The divisor \(D\) is defined as the sum of those toric divisors whose rays share the same sign as the center of \(\Delta\). Such a choice of signs also defines a tropical divisor (with $\FF_2$-coefficients) on the tropical toric variety $\T\Sigma_T$, and consequently determines a class in the tropical homology group
\(
H_{n-1}(\T\Sigma_T; \mathcal{F}_{n-1} \otimes \FF_2).
\)
For the general theory of tropical toric varieties, we refer the reader to our previous work~\cite{MR24} and the references therein.

If we denote by \(\mathbb{R}X_D\) the real part of the hypersurface \(X_D\) obtained by patchworking, we proved in \cite{MR24} that whenever two divisors \(D\) and \(D'\) are linearly equivalent (over \(\mathbb{F}_2\)), the corresponding real hypersurfaces \(\mathbb{R}X_D\) and \(\mathbb{R}X_{D'}\) are identical up to an automorphism of the ambient toric variety \(\mathbb{C}\Sigma_T\). In fact, the sign distributions defining \(D\) and \(D'\) correspond to two octants of the same patchworking. The main result of the second part of \cite{MR24} is the following theorem:
\begin{thm}[\cite{MR24}] \label{thm:conncomp}
Let $\R X_D$ be a real Calabi-Yau hypersurface arising from a central, primitive patchworking of a reflexive polytope $\Delta$ such that the tropical homology groups $H_{n}(X_{\text{trop}}; \mathcal F_k \otimes \FF_2)$ vanish for all $0 < k < n$.  Then $\R X_D$ is connected if and only if the tropical divisor class $D_{| X_{\text{trop}}^{\circ}} \in H_{n-1}(X_{\text{trop}}^{\circ}; \mathcal F_{n-1} \otimes \FF_2)$ is not zero. 
\end{thm}

\subsection{Cohomological patchworking spectral sequence}
In the forthcoming paper \cite{MR26}, we will shift our focus from tropical homology to tropical cohomology. For this purpose we will consider the cohomological counterpart of the patchworking spectral sequence corresponding to the distribution of signs introduced in~\cite{RS}. A detailed construction of this spectral sequence can be found in \cite{chenal2023}. For convenience, we recall here its main properties. The first page is given by the tropical cohomology groups
\[ E^1_{p,q} = H^q(X_{\trop}, \mathcal F^p \otimes \FF_2)\]
The differentials on the $k$-th page are 
\[ \delta^k: E^k_{p,q} \rightarrow E^k_{p-k,q+1}. \]
The spectral sequence converges to the cohomology of the corresponding real variety $\R X$, i.e.
\[ H^q(\R X, \FF_2) \cong \sum_{p} E^{\infty}_{p,q}. \]

\subsection{First Betti number of real Calabi-Yau patchworking}
In  \cite{MR26} we also prove the cohomological analogue of Theorem \ref{thm1MR24}:
\begin{thm}\label{thm:cohommirroriso}
    Let $A$ be a commutative ring. For all $p$ and $q$ there are canonical isomorphisms
\begin{equation} \label{trop_ms}
H^q(X_{\trop}; \mathcal F^p\otimes A) \cong H^q(X^\circ_{\trop}; \mathcal F^{n-p}\otimes A).
\end{equation}
\end{thm} 

Within this framework, we pursue the study of the topology of patchworkings arising from reflexive polytopes equipped with unimodular central triangulations and sign distributions on the vertices of these triangulations. As recalled earlier, such a choice of signs defines a tropical divisor 
\[
D \in H^{1}(\mathbb{T}\Sigma_{T}, \mathcal{F}^{1}\otimes \FF_2),
\]
whose restriction to \( H^{1}(X^\circ_{\trop}, \mathcal{F}^{1}\otimes \FF_2) \) is denoted by \( D\vert_{X^\circ_{\trop}} \).
The main result in~\cite{MR26} is the following theorem:

\begin{thm}[\cite{MR26}]\label{thm:firstBetti}
Let $\delta^1_{D}$ be the differential of the first page of the cohomological patchworking spectral sequence associated to the divisor $D$. Then the following diagram is commutative
\[\begin{tikzcd}
H^{1}(X_{\trop},\mathcal{F}^{n-1}\otimes\FF_2) \arrow{r}{\delta^1_{D}} \arrow{d}{\cong} & H^{2}(X_{\trop},\mathcal{F}^{n-2}\otimes\FF_2) \arrow{d}{\cong} \\
H^{1}(X_{\trop}^{\circ},\mathcal{F}^{1}\otimes\FF_2) \arrow{r}{\cup_D} & H^{2}(X_{\trop}^{\circ},\mathcal{F}^{2}\otimes\FF_2),
\end{tikzcd}
\]
where the vertical maps are the mirror maps from Theorem \ref{thm:cohommirroriso} and
$$
\cup_D(\beta)=\beta^2+D\vert_{X_{\trop}^{\circ}}\cup\beta.
$$
\end{thm}

In certain cases this theorem is sufficient to compute the first Betti number (with $\FF_2$ coefficients) of $\R X_D$. Recall that the real locus $\R X$ is called maximal if equality holds in the Smith-Thom inequality $\sum_{i} b_i(\R X) \leq \sum_{i} b_i(\C X)$.

\begin{cor}\label{cor:firstBetti}
    Suppose that either \(\Delta\) or $\Delta^\circ$ is a non-singular po\-ly\-tope and that there exists a divisor \(D\in H^1(\T\Sigma_T,\mathcal{F}^1\otimes\FF_2)\) such that
    \begin{equation} \label{d1_zero}
        \beta^2 + D\vert_{X_{\trop}^{\circ}} \cup \beta \equiv 0 \pmod{2} \quad \text{for all} \ \  \beta\in H^1(X^\circ_{\trop};\F^1\otimes\FF_2).
    \end{equation}
If the cohomological patchworking spectral sequence associated to the divisor $D$ degenerates at the second page, we have two cases
    \begin{itemize}
        \item if $D|_{X^\circ_{\trop}} =0 $ then \(\mathbb{R}X_D\) has two connected components and 
         \[
        b_1(\mathbb{R}X_D) = h^{n-1,1}(X)+h^{1,1}(X).
        \]
        \item if $D|_{X^\circ_{\trop}} \neq 0$ then \(\mathbb{R}X_D\) is connected and
        \[
        b_1(\mathbb{R}X_D) = h^{n-1,1}(X)+h^{1,1}(X)-1.
        \]
 \end{itemize}
\end{cor}

\begin{proof}
    Suppose that $\Delta$ is non-singular. Then $X_{\trop}$ is combinatorially ample and by the Lefschetz hyperplane theorem in \cite{ARS}, the only non-zero tropical Hodge numbers of $X_{\trop}$ are those of type $(j,j)$ and $(j,n-j)$ and they coincide with the standard Hodge numbers of $X$.
Moreover, it also follows from \cite{ARS} and the universal coefficients theorem, that the tropical cohomology groups over $\Z$ have no torsion, hence also $\dim H^j(X_{\trop}, \mathcal F^k \otimes \FF_2) = h^{j,k}(X)$. By Theorem \ref{thm1MR24} and the classical mirror symmetry identity \eqref{ms_hodge}, the same is true for $X^\circ_{\trop}$, $\dim H^j(X^\circ_{\trop}, \mathcal F^k \otimes \FF_2) = h^{j,k}(X^{\circ})$. 

This implies that the only non-trivial maps in the first page of the spectral sequence which contribute to $b_1(\R X)$ are
\[ 
 H^0(X_{\trop}, \mathcal F^n \otimes \FF^2) \stackrel{\delta^1}{\longrightarrow} H^1(X_{\trop}, \mathcal F^{n-1} \otimes \FF^2) \stackrel{\delta^1}{\longrightarrow} H^2(X_{\trop}, \mathcal F^{n-2} \otimes \FF^2). 
 \]
The first map is zero if and only if $D|_{X^\circ_{\trop}} =0$ by Theorem \ref{thm:conncomp}. The second map is computed via Theorem \ref{thm:firstBetti} and condition \eqref{d1_zero} implies that it is zero. The other term which contributes to $b_1(\R X)$ is $E^1_{1,1} =  H^1(X_{\trop}, \mathcal F^1 \otimes \FF^2)$ and it does not change in the second page. The result then follows since $H^0(X_{\trop}, \mathcal F^n \otimes \FF^2) \cong \Z_2$.
\end{proof}

\begin{cor} \label{cor:firstBetti2}
     If $\Delta$ and $\Delta^\circ$ are $4$ or $5$ dimensional ($n=3$ or $4$) and at least one of them is non-singular, then $E^\infty_{n-1,1} = E^2_{n-1,1}$ and $E^{\infty}_{1,1} = E^2_{1,1}$. Therefore, in these dimensions, if \eqref{d1_zero} holds, we have
     \begin{itemize}
        \item if $D|_{X^\circ_{\trop}} =0 $, then \(\mathbb{R}X_D\) has two connected components and 
         \[
        b_1(\mathbb{R}X_D) = h^{n-1,1}(X)+h^{1,1}(X);
        \]
        \item if $D|_{X^\circ_{\trop}} \neq 0$, then \(\mathbb{R}X_D\) is connected and
        \[
        b_1(\mathbb{R}X_D) = h^{n-1,1}(X)+h^{1,1}(X)-1.
        \]
     \end{itemize}
     In the first case $\R X$ is maximal. 
\end{cor}

\begin{proof}
    The only non trivial higher maps in the patchworking spectral sequence which contribute to the computation of $b_1$ are $\delta^{n-1}: E^{n-1}_{n,0} \rightarrow E^{n-1}_{1,1}$ and $\delta^{n-3}: E^{n-3}_{n-1,1} \rightarrow E^{n-3}_{2,2}$. As for the first map, in the proof of Theorem \ref{thm:conncomp} in 
\cite{MR24}, we showed that it is always zero, in particular $E^{\infty}_{1,1} \cong E^{2}_{1,1} \cong H^1(X_{\trop}, \mathcal F^1 \otimes \FF_2)$. As for the second map, it does not exist for $n=3$ or $4$. Therefore $E^{\infty}_{n-1,1} = E^{2}_{n-1,1}$. 

In the case of the first bullet, for $n=3$, maximality follows from Poincaré duality. In the case $n=4$, we use the general fact that the Euler characteristic of $\R X$ is equal to the signature of $X$ \cite{BB2010}. This can be seen also from the spectral sequence (as observed in \cite{RS}), indeed the signature formula for a $2m$-dimensional K\"ahler manifold gives
\[ \sigma(X) = \sum_{p,q = 0}^{2m} (-1)^q h^{p,q} = \sum_{p,q = 0}^{2m} (-1)^q \dim E^1_{p,q}. \]
On the other hand, the right hand side is the Euler characteristic of the spectral sequence, which is stable on all of its pages. Therefore
\[ \sigma(X)  = \sum_{p,q = 0}^{2m} (-1)^q \dim E^\infty_{p,q} = \chi(\R X). \]
Assuming the first bullet, these formulas imply that 
$$b_2(\R X) = 2 h^{4,0}(X) + h^{2,2}(X)- 2h^{0,0}(X).$$ 
In other words, $\R X$ is maximal. 
\end{proof}

We point out that, in a different context, a  result similar to Theorem \ref{thm:firstBetti} was proved by Arguz-Prince in \cite{AP2020} for the zero divisor case and by the first author in \cite{Matessi2024}. These articles deal with anti-symplectic involutions which preserve a Lagrangian fibration on a $3$-dimensional Calabi-Yau manifold. 

\subsection{Applications}
The main goal of this paper is to give applications of Theorem \ref{thm:firstBetti}, Corollary \ref{cor:firstBetti} and Corollary \ref{cor:firstBetti2} to concrete examples. In Section \ref{sec:cube}, we focus on the case of the \((n+1)\)-cube \(\square = [-1,1]^{n+1}\) and its dual \(\square^\circ\). We first show that, when considering the zero divisor on \(\mathbb{T}\Sigma_{\square}\), the map on the first page of the patchworking spectral sequence for \(X_{\trop}^\circ\) is trivial for all $n$ (see Proposition~\ref{prop:zeromap}). We also examine the case of a non-zero divisor \(D\). In this situation, we show that for all $n>4$, the first Betti number of the patchworking associated with \(D\) coincides with \(h^{\,n-1,1}(X)\), where \(X\) is an anticanonical hypersurface in \((\mathbb{C}\mathbb{P}^1)^{\,n+1}\) (see Proposition~\ref{prop:nonzerodivisor} and Corollary~\ref{cor:Bettinumberdualcube}). In Section~\ref{sec:maxfirstBetti} we address the problem of finding a divisor \(D\) on \(X^\circ_{\trop}\) such that the corresponding map \(\delta^1_D\) in the patchworking spectral sequence for $X_{\trop}$ vanishes. Equivalently, by Theorem \ref{thm:firstBetti}, we seek a divisor \(D\) on $X^{\circ}_{\trop}$ satisfying
\[
\beta^2 + D\cup \beta = 0 \pmod{2}
\]
for every divisor \(\beta\). In Section~\ref{sec:localconf}, we provide sufficient local conditions ensuring the existence of such a divisor. We then present, in Section~\ref{sec:localexample}, explicit local examples of divisors in dimension three that meet these conditions, and finally show in Section~\ref{sec:globalex} that, in certain cases, one can construct global divisors satisfying all of them.
\begin{thm}\label{thm:introlast} Let $\Delta \subset M_{\R}$ be a smooth reflexive polytope of dimension $4$ satisfying one of the following hypothesis
\begin{itemize}
    \item for every $2$-dimensional face $F$ of $\Delta$, the boundary $\partial F$ contains only three parities;
    \item every edge of $\Delta$ has even length.
\end{itemize}
Then, for every primitive central triangulation $T$ of $\Delta$, there exists a divisor $D \in H^1(X_{\trop}^\circ, \mathcal F^1\otimes\FF_2)$ satisfying 
$$
\beta^2+D\cup \beta=0
$$ 
for all divisors $\beta$.
\end{thm}

\begin{cor}
    Under the same assumptions as in Theorem~\ref{thm:introlast}, for any primitive triangulation \(T\) of \(\Delta\) there exists a divisor \(D\) such that the real hypersurface \(\mathbb{R}X_D\) is connected and
    \[
        b_1(\mathbb{R}X_D) = h^{1,2}(X)+h^{1,1}(X)-1.
    \]
\end{cor}
\subsection*{Acknowledgments}
We thank Matthias Franz for pointing out the work \cite{Masuda2006}.  Diego Matessi was partially supported by the national research project ``Moduli spaces and special varieties'' (PRIN 2022) and he is a member of the INDAM research group GNSAGA. Arthur Renaudineau acknowledges the support of the CDP C2EMPI, as well as the French State under the France-2030 programme, the University of Lille, the Initiative of Excellence of the University of Lille, the European Metropolis of Lille for their funding and support of the R-CDP-24-004-C2EMPI project, and also the ANR, project ANR-22-CE40-0014. Both authors wish to thank the FBK-CIRM and the Department of Mathematics at the University of Trento for their support and hospitality during the ``Research in pairs program''.

\section{Tropical mirror pairs and cohomology}
We briefly recall in this section the fundamental definitions from~\cite{MR24}, to which we refer for additional details and illustrative examples. In the following $M \cong \Z^{n+1}$ is a lattice, $N = \Hom(M, \Z)$ is the dual lattice and $M_{\R} = M \otimes \R$. Given a fan $\Sigma$ (either in $M_{\R}$ or $N_{\R}$) we denote by $\C \Sigma$ and $\T \Sigma$ the corresponding complex toric variety or tropical toric variety, respectively. 

\subsection{Reflexive Polytopes and Mirror Symmetry}\label{subsection:mirrordefi}

A lattice polytope $\Delta \subset M_{\R}$ is called \emph{reflexive} if it can be written as
\[
\Delta = \{ u \in M_{\R} \mid \langle v, u \rangle \le 1, \; v \in N \}.
\]
 In particular, this implies that the only interior lattice point of $\Delta$ is the origin:
\[
\operatorname{int}(\Delta) \cap M = \{ 0 \}.
\]
The \emph{dual polytope} $\Delta^{\circ} \subset N_{\R}$ is defined by
\[
\Delta^{\circ} = \{ v \in N_{\R} \mid \langle v, u \rangle \le 1, \; \forall u \in \Delta \}.
\]
If $\Delta$ is reflexive, then its dual $\Delta^{\circ}$ is also a reflexive lattice polytope, and moreover $(\Delta^{\circ})^{\circ} = \Delta$.

A reflexive polytope $\Delta \subset M_{\R}$ defines a fan in the same space, called its \emph{face fan}, denoted by $\Xi_{\Delta}$, whose cones are generated by the faces of $\Delta$. These fans satisfy
\[
\Xi_{\Delta} = \Sigma_{\Delta^{\circ}} \quad \text{and} \quad \Xi_{\Delta^{\circ}} = \Sigma_{\Delta},
\]
where $\Sigma_{\Delta}$ and $\Sigma_{\Delta^\circ}$ denote the respective normal fans.
\begin{defi}
A lattice simplex is called \emph{primitive} (or \emph{unimodular}) if its volume equals $\frac{1}{(n+1)!}$, the minimal possible volume among lattice simplices.
\end{defi}

If $\Delta$ is a non-singular reflexive polytope, the cones of $\Xi_{\Delta^{\circ}}$ are regular. Consequently, all faces of $\Delta^{\circ}$ are simplices with no interior lattice points. Moreover, if $\sigma$ is a facet of $\Delta^{\circ}$, then the convex hull of the origin in $N$ together with $\sigma$ forms a primitive simplex.

Given a pair of dual reflexive polytopes $(\Delta, \Delta^{\circ})$, one may consider the \emph{anticanonical hypersurfaces}
\[
X \subset \C\Sigma_{\Delta}, \qquad X^{\circ} \subset \C\Sigma_{\Delta^{\circ}}.
\]
Both $X$ and $X^{\circ}$ are Calabi--Yau varieties, forming a \emph{mirror pair}. Typically, both the hypersurfaces and their ambient toric varieties are singular. In favorable situations, these singularities can be resolved by a \emph{crepant resolution}, which can often be obtained through \emph{primitive convex central subdivisions}.

\begin{defi}
Let $\Delta \subset M_{\R}$ be a reflexive polytope.  
A triangulation of $\Delta$ with vertices in $M$ is called \emph{central} if it is obtained by taking the convex hulls of the origin in $M$ together with the simplices from a triangulation of the facets of $\Delta$.  
It is called \emph{convex} if it is induced by an integral convex piecewise linear function on $\Delta$.  
The triangulation is said to be \emph{primitive} if all of its simplices are primitive.  

We denote such a triangulation by $T$, and by $\partial T$ the collection of simplices of $T$ contained in the boundary $\partial \Delta$.
\end{defi}

Let $T$ be a primitive central triangulation of $\Delta$. The corresponding \emph{regular fan} $\Sigma_T$ is constructed by taking the cones from the origin of $M$ over the simplices in the subdivision of $\Delta$. This fan refines $\Xi_{\Delta} = \Sigma_{\Delta^{\circ}}$, and therefore defines a desingularization
\[
\C\Sigma_T \longrightarrow \C\Sigma_{\Delta^{\circ}}.
\]

Throughout this work, we assume that there exist primitive triangulations $T$ and $T^{\circ}$ of $\Delta$ and $\Delta^{\circ}$ respectively; convexity will not always be assumed. In this situation, generic anticanonical hypersurfaces in $\C\Sigma_T$ and $\C\Sigma_{T^{\circ}}$ are smooth Calabi--Yau varieties.  
We denote by
\[
X \subset \C\Sigma_{T^{\circ}}, \qquad X^{\circ} \subset \C\Sigma_T
\]
the smooth anticanonical hypersurfaces replacing the singular ones in $\C\Sigma_{\Delta}$ and $\C\Sigma_{\Delta^{\circ}}$.  
When the triangulations are not convex, these varieties may be non-projective.

In the projective case (that is, when both $T$ and $T^{\circ}$ are convex), Batyrev and Borisov~\cite{BB} proved that the Hodge numbers of $X$ and $X^{\circ}$ satisfy the \emph{mirror symmetry identity}
\begin{equation} \label{ms_hodge}
    h^{p,q}(X) = h^{n-p,q}(X^{\circ}).
\end{equation}

This relation is established in Theorem~4.15 of~\cite{BB} for the \emph{stringy Hodge numbers}, which depend only on the singular Calabi--Yau varieties.  
Proposition~1.1 of~\cite{BB} shows that, in the smooth projective case, these coincide with the usual Hodge numbers of the crepant resolutions determined by $T$ and $T^{\circ}$.  
Although a general proof does not appear to our knowledge in the literature, it is expected that the mirror identity of Hodge numbers continues to hold in the non-projective case.

\subsection{Tropical mirror pairs}\label{combmirror}

We first introduce some notations that will be used throughout the paper.  
Given a simplex $\sigma \in T$, we define:
\begin{itemize}
    \item $C(\sigma) \in \Sigma_T$ as the cone over $\sigma$;
    \item $\langle \sigma \rangle \subset M$ denotes the integral tangent space to $\sigma$;
    \item $\sigma^{\perp}$ as the subspace of $N$ orthogonal to $\langle \sigma \rangle$;
    \item if $\sigma \neq 0$, then 
    \[
        \sigma_{\infty} = \sigma \cap \partial \Delta.
    \]
\end{itemize}

Given a cone $\rho \in \Sigma_T$, we define:
\begin{itemize}
    \item $S(\rho) = \rho \cap \Delta$;
    \item $\hat{\sigma}$ as the simplex $S(C(\sigma))$, i.e. the convex hull of $0$ and $\sigma$.
\end{itemize}

In particular, when $0 \in \sigma$, we have $\hat{\sigma} = \sigma$, and $\sigma_{\infty}$ is the unique facet of $\sigma$ contained in $\partial \Delta$.  
If $0 \notin \sigma$, then $\sigma_{\infty} = \sigma$.

\medskip
Given two fans $\Sigma$ and $\Sigma'$ such that $\Sigma$ is a refinement of $\Sigma'$, we set:
\begin{itemize}
    \item for every cone $\rho \in \Sigma$, denote by $\min(\rho)$ the smallest cone of $\Sigma'$ containing $\rho$.
\end{itemize}

If $\Sigma$ is the normal fan of a polytope $\Delta$, we denote:
\begin{itemize}
    \item $\rho^{\vee}$ the face of $\Delta$ normal to $\rho$;
    \item for any face $F$ of $\Delta$, $F^{\vee}$ the cone of $\Sigma$ normal to $F$.
\end{itemize}

\medskip
Assume that $T$ is a primitive convex central triangulation of $\Delta$.  
If $\mu : \Delta \cap M \to \mathbb{Z}_{\ge 0}$ is a convex function ensuring the convexity of $T$, one can consider its Legendre transform (defined over $N_{\mathbb{R}}$):
\begin{equation}\label{tropicalhypersurface}
    \mu^*(x) = \max_{y \in \Delta \cap M} \big( \langle x, y \rangle - \mu(y) \big),
\end{equation}
and define the set
\[
    V_\mu := \left\{ x \in N_{\mathbb{R}} \ \middle|\ 
        \exists\, y_1 \neq y_2 \text{ such that } 
        \mu^*(x) = \langle x, y_1 \rangle - \mu(y_1)
                  = \langle x, y_2 \rangle - \mu(y_2) 
    \right\}.
\]
In other words, $V_\mu$ is the set of points where the maximum in~\eqref{tropicalhypersurface} is attained at least twice.  
It is the \emph{tropical hypersurface} associated to $\mu$: a weighted, rational, balanced polyhedral complex inducing a subdivision of $N_{\mathbb{R}}$ which is dual (as a poset) to the triangulation $T$ (see for example~\cite{BIMS}).  

Moreover, the compactification of $V_\mu$ in the tropical toric variety $\mathbb{T}\Sigma_\Delta$ induces a subdivision of $\mathbb{T}\Sigma_\Delta$ whose underlying poset is the subposet of $\Sigma_\Delta \times T$:
\begin{equation} \label{poset:delta}
    \Phi := \left\{ (\rho, \sigma) \in \Sigma_\Delta \times T \ \middle|\ \sigma \subset \rho^\vee \right\}.
\end{equation}
The order on this poset is the inverse of inclusion, i.e.
$(\rho, \sigma) \leq (\rho', \sigma')$ iff $\rho' \subset \rho$ and $\sigma' \subset \sigma$
(see~\cite{brugalle2022combinatorial}, Remark~2.3).

\begin{defi} \label{dual:trop}
Let $T$ be a primitive convex central triangulation of $\Delta$ and $T^\circ$ a primitive convex central triangulation of $\Delta^\circ$, induced respectively by convex functions $\mu$ and $\mu^\circ$.  
We denote by $X_{\mathrm{trop}}$ the compactification of $V_\mu$ inside $\mathbb{T}\Sigma_{T^\circ}$, and by $X_{\mathrm{trop}}^\circ$ the compactification of $V_{\mu^\circ}$ inside $\mathbb{T}\Sigma_T$. The pair $(X_{\trop},X_{\trop}^\circ)$ is called a tropical mirror pair.
\end{defi}

The tropical variety $X_{\mathrm{trop}}$ induces a subdivision of $\mathbb{T}\Sigma_{T^\circ}$ with underlying poset
\[
    \mathcal{P}(T^\circ, T) :=
    \left\{
        (\tau, \sigma) \in T^\circ \times T
        \ \middle|\ 
        0 \in \tau \text{ and } \sigma \subset \min(C(\tau))^\vee
    \right\}.
\]
Recall that $C(\tau)$ is the cone (in $\Sigma_{T^\circ}$) over $\tau$.  
Since $\Sigma_{T^\circ}$ is a refinement of the fan $\Xi_{\Delta^\circ}$ (the normal fan of $\Delta$), we have that $\min(C(\tau))^\vee$ is a face of $\Delta$.  
The order on $\mathcal{P}(T^\circ, T)$ is again the inverse of inclusion.
We will be intersted more specifically in the subposet of $\mathcal{P}(T^\circ,T)$ defined by 
$$
\mathcal{P}^1(T^\circ,T):=\left\lbrace (\tau,\sigma)\in \mathcal{P}(T^\circ,T) \vert \dim(\sigma)\geq 1 \right\rbrace .
$$
This poset parametrizes $X_{\text{trop}}$. As it will appear later, this poset will be the support of the cosheaves we will consider.
We refer the reader to our paper~\cite{MR24} for illustrative examples. There, we address the more general setting of subdivisions~$T$ and~$T^\circ$ that are not necessarily convex, a case we omit here for reasons of clarity and conciseness. We now make a remark that will be useful later.
\begin{rem}\label{rem:divinfacet}
Let $\tau$ be an edge of $T^\circ$ containing the origin, whose other endpoint lies in the interior of a facet of $\Delta^\circ$.
Then there is no element of $\mathcal{P}^1(T^\circ, T)$ whose first coordinate is $\tau$, since $\min C(\tau)^{\vee}$ is zero dimensional. Geometrically, this means that the divisor of $\T\Sigma_{T^\circ}$ corresponding to $\tau$ does not intersect the tropical hypersurface $X_{\trop}$.
\end{rem}

\subsection{Tropical cohomology}
Tropical cohomology introduced by Itenberg, Katzarkov, Mikhalkin and Zharkov \cite{trop_hom}, in its cellular version, belongs to the realm of poset cohomology (see for example \cite{brugalle2022combinatorial} and reference therein for poset cohomology). 
We fix two primitive central triangulations $T$ and $T^\circ$ of $\Delta$ and $\Delta^\circ$, respectively, which are not necessarily convex.  
The poset $\mathcal{P}(T^\circ,T)$ is graded by
\[
    \dim(\tau,\sigma) = \operatorname{codim}(\tau) - \dim(\sigma).
\]
Note that in the geometric realization of $\mathcal{P}(T^\circ,T)$, we have
\[
    \dim(\tau,\sigma) = \dim F(\tau,\sigma).
\]
Moreover, the poset $\mathcal{P}(T^\circ,T)$ is \emph{thin} and admits a \emph{balanced signature}, obtained by choosing orientations on each cell $F(\tau,\sigma)$.

\medskip
Recall that the $p$-multitangent cosheaves over $\mathcal{F}_p(T^\circ,T)$ are defined by
$$
\mathcal{F}_p(\tau,\sigma):=\sum_{
\begin{array}{c}
\lambda\subset\sigma \\ \dim(\lambda)=1
\end{array}} \bigwedge^p (\lambda^\perp / \langle C(\tau) \rangle ).
$$
The $p$-multitangent sheaves $\mathcal{F}^p(T^\circ,T)$ are the duals
$$
\mathcal{F}^{p}(\tau,\sigma):=\mathrm{Hom}(\F_p(\tau,\sigma),\Z).
$$

Let $s=\dim \sigma$ and let $\vol_{\sigma}$ be a generator of $\wedge^s \langle \sigma \rangle$, unique up to a sign. Then we have
\begin{lem}\label{lem:tropsheaf}

\[ \mathcal{F}^{p}(\tau,\sigma) = \frac{\bigwedge^p C(\tau)^\perp}{\operatorname{vol}_{\sigma} \wedge \bigwedge^{p-\dim \sigma} C(\tau)^\perp }. \]

\end{lem}

\begin{proof}
For every $\tau \in T^\circ$ we will denote 
 $$(N)_{\tau} = N  / \langle C(\tau) \rangle,$$
 where $\langle C(\tau) \rangle$ is the submodule of $N$ generated by $C(\tau) \cap N$.
    We have $\Hom(\bigwedge^p N_\tau, \Z)= \bigwedge^p C(\tau)^\perp$. Therefore, 
    \begin{equation*}
        \begin{split}
            \mathcal{F}^{p}(\tau,\sigma) & = \frac{ \bigwedge^p C(\tau)^\perp}{[ \sum_{
\lambda\subset\sigma, \dim(\lambda)=1 }
 \bigwedge^p (\lambda^\perp / \langle C(\tau) \rangle ) ]^\perp} = \\
 \ &=\frac{ \bigwedge^p C(\tau)^\perp}{\bigcap_{
\lambda\subset\sigma, \dim(\lambda)=1 }
 [\bigwedge^p (\lambda^\perp / \langle C(\tau) \rangle )]^\perp}.
        \end{split}
    \end{equation*}
For any submodule $W \subset N_\tau$ we have that 
\[ \left[\bigwedge^p W \right]^\perp = W^\perp \wedge \bigwedge^{p-1} C(\tau)^\perp \]
In particular 
 \[ \mathcal{F}^{p}(\tau,\sigma) = \frac{ \bigwedge^p C(\tau)^\perp}{\bigcap_{
\lambda\subset\sigma, \dim(\lambda)=1 }
 (\langle \lambda \rangle \wedge \bigwedge^{p-1} C(\tau)^\perp)}.\]
 We also have 
 \[\bigcap_{
\lambda\subset\sigma, \dim(\lambda)=1 }
 (\langle \lambda \rangle \wedge \bigwedge^{p-1} C(\tau)^\perp) = \vol_{\sigma} \wedge \bigwedge^{p-\dim \sigma} C(\tau)^\perp.\]
 This proves the lemma.
 
\end{proof}

The sheaf maps of $\mathcal{F}^p$ are induced by the natural injections
\begin{equation}\label{cosheafmaps}
   \langle C(\tau') \rangle
    \longrightarrow
    \langle C(\tau) \rangle,
\end{equation}
whenever $\tau \subset \tau'$ (recall that, by definition of $\mathcal{P}(T^\circ,T)$, both $\tau$ and $\tau'$ contain $0$).

\medskip
The support of the sheaves $\mathcal{F}^p$ is the subposet $\mathcal{P}^1(T^\circ,T)$.  
For brevity, if $\mathcal{G}$ is any sheaf on $\mathcal{P}^1(T^\circ,T)$, we will write
\[
    H^\bullet(X_{\trop}; \mathcal{G})
    := H^\bullet(\mathcal{P}^1(T^\circ,T); \mathcal{G})
\]
for the corresponding cohomology groups.

\section{Picard group of the mirror}
Let $(X_{\trop},X^\circ_{\trop})$ be a tropical mirror pair arising from central primitive triangulations $T$ and $T^\circ$ of $\Delta$ and $\Delta^\circ$ of dimensions $n+1$. The main result of this section is the following theorem:
\begin{thm}\label{thm:surj}
    If the polytope $\Delta$ is non-singular and if $n\geq 3$, the map $H^{1}(\T\Sigma_{T},\mathcal{F}^{1}\otimes \FF_2)\rightarrow H^{1}(X_{\trop}^{\circ},\mathcal{F}^{1} \otimes \FF_2)$ induced by inclusion is surjective.
\end{thm}
\begin{rem}
In the case where the polytope $\Delta^\circ$ is non-singular, then $\T\Sigma_{T}=\T\Sigma_{\Delta^\circ}$ and the hypersurface $X^\circ_{\trop}$ in $\T\Sigma_{T}$ is combinatorially ample (see Definition 2.5 in \cite{ARS}). Then results from \cite{ARS} apply and the map in Theorem \ref{thm:surj} is an isomorphism.  
\end{rem} 

\begin{rem}
In \cite{TMS}, Gross studies a mirror \(X^\circ\) of the quintic threefold \(X \subset \mathbb{P}^4\), obtained by choosing a particular triangulation of the associated polytope. He examines a collection of divisors on \(X^\circ\) and shows in Lemma~4.3 that these divisors generate \(H^2(X^\circ,\mathbb{Z})\). Notably, the divisors in question arise precisely from the ambient toric variety. Consequently, in this particular case, Theorem~\ref{thm:surj} may be viewed as a tropical analogue of Lemma~4.3 of \cite{TMS}

\end{rem}

The strategy to prove Theorem \ref{thm:surj} is the same as in \cite{ARS}. For simplicity we write $\mathcal F^p$ instead of $\mathcal F^p \otimes \FF^2$. We introduce the following short exact sequences of cellular sheaves:
    \begin{equation}\label{eq:Qp}
          0 \rightarrow \mathcal{Q}^p\rightarrow \mathcal{F}^p_{\Delta^{\circ}}\rightarrow\mathcal{F}^p_{\Delta^{\circ}}\vert_{X^{\circ}_{\trop}}\rightarrow 0,  
    \end{equation}

and
    \begin{equation}\label{eq:Np}
    0 \rightarrow \mathcal{N}^p\rightarrow \mathcal{F}^p_{\Delta^{\circ}}\vert_{X^{\circ}_{\trop}}\rightarrow \mathcal{F}^p_{X^{\circ}_{\trop}}\rightarrow 0.
    \end{equation}
    Here the sheaf $\mathcal{F}^p_{\Delta^{\circ}}$ is defined on the subdivision of $\T\Sigma_T$ induced by $X^{\circ}_{\trop}$ with underlying poset the subposet of $T\times T^\circ$ defined by 
    $$
    \mathcal{P}(T,T^\circ):=\left\lbrace(\tau,\sigma)\in T\times T^\circ \mid 0\in\tau \mbox{ and } \sigma\in\min(C(\tau))^\vee \right\rbrace.
    $$
    Its restriction to $X^{\circ}_{\trop}$ is the restriction to the subposet
    $$
    \mathcal{P}^1(T,T^\circ)=\left\lbrace(\tau,\sigma)\in \mathcal{P}(T,T^\circ) \mid \dim(\sigma)\geq 1 \right\rbrace.
    $$
    The sheaf $\mathcal{F}^p_{X^{\circ}_{\trop}}$ is also defined on the poset $\mathcal{P}^1(T,T^\circ)$. The sheaf $\mathcal{F}^p_{\Delta^{\circ}}$ is simply defined by 
    $$
    \mathcal{F}^{p}_{\Delta^{\circ}}(\tau,\sigma):=\bigwedge^p C(\tau)^\perp.
    $$
    Moreover, recall that Lemma \ref{lem:tropsheaf} states that
    \[ \mathcal{F}^{p}_{X^{\circ}_{\trop}}(\tau,\sigma) = \frac{\bigwedge^p C(\tau)^\perp}{\operatorname{vol}_{\sigma} \wedge \bigwedge^{p-\dim \sigma} C(\tau)^\perp }. \]
    \begin{prop}\label{prop:vanishQ}
        One has $H^2(\mathcal{Q}^1)=0$.
    \end{prop}
\begin{proof}

Since the polytope $\Delta$ is non-singular, the set of points in $\Delta^\circ\cap N$ coincide with the set of vertices of $\Delta^\circ$. Moreover, for any vertex $v$ of $\Delta^\circ$, the set $\min(C(v))$ consists only of the cone $C(v)$ over $v$ and its dual face $(\min(C(v)))^\vee$ is the facet of $\Delta$ normal to $C(v)$. So the support of $Q^p$ is 
$$
\bigsqcup_{v\in \mathrm{Vert}(\Delta^\circ)} \left\lbrace (\tau,v) \mid \tau_\infty \in C(v)^\vee\right\rbrace \subset \mathcal{P}(T,T^\circ)\setminus \mathcal{P}^1(T,T^\circ).
$$
Denote by $\Sigma_{T_v}$ the fan in $M_\R$ consisting of all cones $C(\tau)$ such that $(\tau,v)\in\mathcal{P}(T,T^\circ)$. The fan $\Sigma_{T_v}$ is obtained by taking cones over the triangulation of the facet $C(v)^\vee$ of $\Delta$. The support of $\mathcal{Q}^p$ is then 
$$
\bigsqcup_{v\in \mathrm{Vert}(\Delta^\circ)} | \Sigma_{T_v}| \times v,
$$
and moreover
$$
\mathcal{Q}^p\vert_{\Sigma_{T_v}\times v}=\mathcal{F}^p\vert_{\Sigma_{T_v}}.
$$
Then
$$
H^q(Q^p)=\bigoplus_{v\in \mathrm{Vert}(\Delta^\circ)} H^q_c(\T\Sigma_{T_v},\mathcal{F}^p).
$$
As explained for example in \cite{jordan1998}[Section 2.5], the compactly supported cohomology groups $H^i_c(\C \Sigma_{T_v},\FF_2)$ splits as a (non-canonical) direct sum 
$$
H^i_c(\C \Sigma_{T_v},\FF_2)=\bigoplus_{p+q=i}H_c^q(\T\Sigma_{T_v},\mathcal{F}^p).
$$
 Since the fan $\Sigma_{T_v}$ is non-singular, by Poincaré duality and universal coefficient theorem, one has $H^{i}_c(\C \Sigma_{T_v},\FF_2)\cong H^{2n+2-i}(\C \Sigma_{T_v},\FF_2)$.
Since $\Sigma_{T_v}$ is full dimensional with convex support, the cohomology 
$H^{k}(\C \Sigma_{T_v}, \C)$ vanishes for odd $k$ (see, for example, 
\cite[Corollary~4.4]{oda1993}). In fact, this vanishing holds also with integral 
coefficients (by \cite[Lemma~4.1 and Theorem~4.3]{Masuda2006}). The three 
hypotheses of \cite[Theorem~4.3]{Masuda2006} are satisfied: the variety 
$\C \Sigma_{T_v}$ is locally standard because it is smooth; its orbit space is a 
(non-compact) polytope homeomorphic to $\T\Sigma_{T_v}$, so every face is acyclic; 
and the third hypothesis is also fulfilled. In particular, we obtain 
$H^{2}(Q^{1}) = 0$.

\end{proof}
\begin{prop}\label{prop:vanishN}
    One has $H^2(\mathcal{N}^1)=0$.
\end{prop}
\begin{proof}
It follows from Lemma \ref{lem:tropsheaf} that 
$$
\mathcal{N}^p(\tau,\sigma)=\vol_{\sigma}\wedge\bigwedge^{p-\dim\sigma}C(\tau)^\perp.
$$
In particular, for $p=1$, one obtains that $\mathcal{N}^1(\tau,\sigma)=\langle \sigma \rangle$ if $\dim(\sigma)=1$ and $\mathcal{N}^1(\tau,\sigma)=0$ otherwise. Note that since the polytope $\Delta$ is non-singular, $\min C(e)=C(e)$ for every $e \in T^{\circ}$ and $C(e)^{\vee}$ is an $n$-dimensional face of $\Delta$. 
Then the support of $\mathcal{N}^1$ is 
$$
\bigsqcup_{e\in \mathrm{Edge}(T^\circ)} \left\lbrace (\tau,e) \mid \tau_\infty \in C(e)^\vee\right\rbrace \subset \mathcal{P}^1(T,T^\circ).
$$
Denote by $\Sigma_{T_e}$ the fan in $M_\R$ consisting of all cones $C(\tau)$ such that $(\tau,e)\in\mathcal{P}^1(T,T^\circ)$. Notice that the support of $\Sigma_{T_e}$ is convex and $n$-dimensional, in particular it is contractible. The support of $\mathcal{N}^1$ is 
$$
\bigsqcup_{e\in \mathrm{Edge}(T^\circ)} | \Sigma_{T_e} | \times e,
$$
and moreover
$$
\mathcal{N}^1\vert_{\Sigma_{T_e}\times e}=\langle e \rangle \cong \FF_2
$$
is a constant sheaf. Then 
$$
H^q(\mathcal{N}^1) \cong \bigoplus_{e\in \mathrm{Edge}(T^\circ)} H^q_c(\T\Sigma_{T_e},\FF_2).
$$
But $H^q_c(\T\Sigma_{T_e},\FF_2)=0$ if $q<\dim(n)$. So in particular if $n\geq 3$ then $H^2(\mathcal{N}^1)=0$.
\end{proof}
\begin{proof}[Proof of Theorem \ref{thm:surj}]
It follows immediately by considering the long exact sequences in cohomology associated to the short exact sequences \ref{eq:Qp} and \ref{eq:Np} and applying Proposition \ref{prop:vanishQ} and \ref{prop:vanishN}.
\end{proof}

\section{First Betti number from the dual of the cube
}\label{sec:cube}

\subsection{Real Calabi-Yau hypersufaces in the dual of the cube} Let $\square=\left[-1,1\right]^{n+1}$, and $\square^\circ$ its dual. Let $(X_{\trop},X^\circ_{\trop})$ be a pair of mirror tropical hypersurfaces arising from central primitive triangulations $T$ and $T^\circ$ of $\square$ and $\square^\circ$. We consider a patchworking on $(\square^{\circ}, T^{\circ}$) and we want to compute the first Betti number of the corresponding real variety $\R X^{\circ}$.

As a preliminary step, we establish a straightforward lemma about the square map on the toric variety $\T\Sigma_{\square}$:
\begin{lem} \label{lemprodcube}
   The map 
$$   \begin{array}{cc}
         Sq : & H^1(\T\Sigma_{\square},\mathcal{F}^1\otimes\FF_2) \rightarrow H^2(\T\Sigma_{\square},\mathcal{F}^2\otimes\FF_2) \\
         & \beta \mapsto \beta \cup \beta
    \end{array}.
    $$
    is identically zero. Moreover, for any $D\in H^1(\T\Sigma_{\square},\mathcal{F}^1\otimes\FF_2)$, the map 
$$    \begin{array}{ccc}
      \cup_D:   H^1(\T\Sigma_{\square},\mathcal{F}^1\otimes\FF_2) & \rightarrow & H^2(\T\Sigma_{\square},\mathcal{F}^2\otimes\FF_2)   \\
       \beta  & \mapsto & \beta\cup D
    \end{array}
    $$
    has kernel generated by $D$ and rank $n$. 
\end{lem} 
\begin{proof}
The rays of the fan $\Sigma_\square$ are generated by the vectors $\pm e_i$ for $1 \leq i \leq n+1$, where $\{e_1, \ldots, e_{n+1} \}$ is the standard basis in $\Z^{n+1}$. 
Let $\beta = D_{\varepsilon e_i}$ be a divisor on $\T\Sigma_\square$ with $\varepsilon = \pm 1$. 
By considering the monomial $m = e_i^*$, we obtain the relation 
\[
D_{e_i} + D_{-e_i} = 0.
\]
It follows that 
\[
D_{\varepsilon e_i}^2 = D_{\varepsilon e_i} D_{-\varepsilon e_i} = 0,
\]
since $e_i$ and $-e_i$ do not belong to a common maximal cone of $\Sigma_\square$. 
Working over $\FF_2$, we have
\[
\left( \sum_i \varepsilon_i D_{e_i} \right)^2 
= \sum_i \varepsilon_i D_{e_i}^2,
\]
and the first assertion follows.

Moreover, since $D_{e_i} = D_{-e_i}$ over $\FF_2$, any divisor can be written as 
\[
D = \sum_{i \in I} D_{e_i}
\]
for some subset $I \subset \{1, \ldots, n+1\}$. 
Let $D = \sum_{i \in I} D_{e_i}$ and $\beta = \sum_{j \in J} D_{e_j}$. 
The condition $D \cup \beta = 0$ implies immediately that $I = J$. 
This establishes the second claim.
\end{proof}

\subsection{The zero divisor} Consider first the zero divisor in $H^1(\T\Sigma_{T^\circ},\mathcal{F}^1\otimes\FF_2)$, and denote by $\delta^1$
the differential of the first page of the associated patchworking spectral sequence. Then
\begin{prop}\label{prop:zeromap}
    For any $n \geq 3$, the map
    $$
    \delta^1:H^1(X^\circ_{\trop},\mathcal{F}^{n-1}\otimes\FF_2)\rightarrow H^2(X^\circ_{\trop},\mathcal{F}^{n-2}\otimes\FF_2)
    $$
    is zero.
\end{prop}
\begin{proof}
    Since we are considering the zero divisor, the mirror map associated to $\delta^1$ is simply the square map
    $$
    \begin{array}{cc}
         Sq : & H^1(X_{\trop},\mathcal{F}^1\otimes\FF_2) \rightarrow H^2(X_{\trop},\mathcal{F}^2\otimes\FF_2) \\
         & \beta \mapsto \beta \cup \beta
    \end{array}.
    $$
    By the cohomological version of tropical Lefschetz hyperplane theorem \cite{ARS}, if $n\geq 3$ one has the following commutative diagram
\[\begin{tikzcd}
H^{1}(\T\Sigma_\square,\mathcal{F}^{1}\otimes\FF_2) \arrow{r}{\mathrm{Sq}} \arrow{d} & H^{2}(\T\Sigma_\square,\mathcal{F}^{2}\otimes\FF_2) \arrow{d} \\
H^{1}(X_{\trop},\mathcal{F}^{1}\otimes\FF_2) \arrow{r}{\mathrm{Sq}} & H^{2}(X_{\trop},\mathcal{F}^{2}\otimes\FF_2),
\end{tikzcd}
\]
where the vertical arrow on the left is an isomorphism.  Since the upper horizontal map is zero by Lemma \ref{lemprodcube}, also the lower horizontal map must be zero. This shows, by Theorem \ref{thm:firstBetti}, that $\delta^1=0$.
\end{proof}

\begin{rem}
    The previous proposition is trivially true for $n=1$. For $n=2$, since we have the equality $\chi(\R X^\circ) = \sigma(X^\circ)$, the first Betti number of $\R X$ is uniquely determined by number of connected components, i.e. in the case of the zero divisor $b_1(\R X^\circ) = h^{1,1}(X^\circ)$. This implies that $\delta_1 = 0$. 
\end{rem}

Applying Corollaries \ref{cor:firstBetti} and \ref{cor:firstBetti2} to this case, we have

\begin{cor}
    If the patchworking spectral sequence associated to the zero divisor degenerates at the second page, then the patchwork $\R X^\circ$ has two connected components and its first Betti number is equal to $h^{1,1}(X)+h^{n-1,1}(X)$, where $X\subset (\C P^1)^{n+1}$ is an anticanonical hypersurface.
\end{cor}

\begin{cor}
    In dimension $n=3$ or $4$ the patchwork $\R X^\circ$ has two connected components and its first Betti number is equal to $h^{1,1}(X)+h^{n-1,1}(X)$, where $X\subset (\C P^1)^{n+1}$ is an anticanonical hypersurface. In other words, in dimension $n=3$ or $4$, $\R X^\circ$ is maximal. 
\end{cor}

\begin{rem}
    In \cite{AP2023} Arguz and Prince, find various four dimensional reflexive polytopes where $\delta^1 = 0$ for the zero divisor, in the context of anti-simplectic involutions preserving a Lagrangian fibration.
\end{rem}

\subsection{Non-zero divisors} Let us now consider the case of a non-zero divisor $D$.
\begin{prop}\label{prop:nonzerodivisor}
    If $n> 4$, the map
$$ 
    \begin{array}{ccc}
      \cup_D:   H^1(X_{\trop},\mathcal{F}^1) & \rightarrow & H^2(X_{\trop},\mathcal{F}^2)   \\
       \beta  & \mapsto & \beta^2+\beta\cup D
    \end{array}
$$
has kernel equal to $\langle D \rangle $.
\end{prop}
\begin{proof}
    It follows again from tropical Lefschetz hyperplane theorem in cohomology and Lemma \ref{lemprodcube}.
\end{proof}
\begin{cor}\label{cor:Bettinumberdualcube}
If $n>4$ and $D\neq 0$, then 
$$
b_1(\R X_{D}^\circ)=h^{n-1,1}(X),
$$
where $X$ is an anticanonical hypersurface in $(\C P^1)^{n+1}$.
\end{cor}
\begin{proof}
    By Proposition \ref{prop:nonzerodivisor} and Theorem \ref{thm:conncomp}, the $\FF_2$ vector spaces $E^2_{n-1,1}$ and $E^2_{n,0}$ of the second page of the patchworking spectral sequence are zero. Since $E^2_{n,0}=0$, the higher map in the patchworking spectral sequence $E^{n-1}_{n,0}\rightarrow E^{n-1}_{1,1}$ is zero, and so $$b_1(\R X_{D})=\dim H^1(X^\circ_{\trop};\F^1\otimes\FF_2).$$
    Moreover by Theorem \ref{thm:cohommirroriso}, one has
    $$
    H^1(X^\circ_{\trop};\F^1\otimes\FF_2)\cong H^1(X_{\trop};\F^{n-1}\otimes\FF_2).
    $$
    Since the toric variety $\Sigma_\square$ is non-singular, the tropical hypersurface $X_{\trop}$ is combinatorially ample and by \cite{ARS} 
    $$
    \dim H^1(X_{\trop};\F^{n-1}\otimes\FF_2)=h^{n-1,1}(X).
    $$ \end{proof}

\begin{rem} In the cases $n=3$ we argue that the map $\beta \mapsto \beta^2+\beta\cup D$ is always zero. We give a non-rigorous proof as follows. 
First of all the map is linear in $D$, hence it is enough to prove it assuming $D = D_{e_j}$ (notation as in proof of Lemma \ref{lemprodcube}). Since
$\beta^2=0$ for all $\beta$, it is enough to assume $\beta = D_{e_k}$ and prove that $(D_{e_j} \cup D_{e_k})_{|X_{\trop}} = 0$ for all $j$ and $k$. By 
Poincaré duality, we can prove this by showing that $(D_{e_j} \cup D_{e_k})_{|X_{\trop}} \cup \alpha = 0$ for all $\alpha \in H^{1}(X_{\trop}, \mathcal F^{1} \otimes \FF_2)$. 
Again, by the Lefschetz theorem for $n=3$, 
\[ H^{1}(X_{\trop}, \mathcal F^{1} \otimes \FF_2) \cong H^{1}(\T\Sigma_{\square}, \mathcal F^{1} \otimes \FF_2). \] 
In particular it is enough to show that $(D_{e_j} \cup D_{e_k} \cup D_{e_m})_{|X_{\trop}} =0$ for all $j,k,m$. 
Such classes, before restricting them to $X_{\trop}$, are either $0$ or correspond in $H^{1}(\T\Sigma_{\square}, \mathcal F^{1} \otimes \FF_2)$ to the Poincaré dual of an edge of the cube $\square$ (this happens if all the indices are distinct). 
An edge intersects $X_{\trop}$ transversely in two points, just like over $\C$,  the line $\PP^1 \times \{0\} \ldots \times \{ 0 \}$ intersects $X$ twice.  Since we are over $\FF_2$, this should imply that the restriction to $X_{\trop}$ is zero. The only reason why this is not a rigorous proof it that we do not know if, in tropical cohomology, the Poincaré dual of a transverse intersection of tropical cycles is equal to the cup product of the Poincaré duals. If this argument is correct, we would have that $\R X^\circ$ is connected and its first Betti number is $h^{1,1}(X) + h^{n-1,1}(X) - 1$, as follows from Corollary \ref{cor:firstBetti2}.
\end{rem}

\section{Maximizing the first Betti number \label{sec:maxfirstBetti}}
\subsection{Maximizing the first Betti number in dimension 3}
Let $(X_{\trop}, X_{\trop}^\circ)$ be a pair of $3$-dimensional mirror tropical hypersurfaces arising from central primitive triangulations $(\Delta, T)$ and $(\Delta^\circ, T^{\circ})$ of dual pairs of $4$-dimensional reflexive polytopes.  We try to maximize the first Betti number of a patchworking $\R X$ associated to a divisor $D \in H^1(X^\circ_{\trop}, \mathcal F^1 \otimes \FF_2)$. As we saw in Corollary \ref{cor:firstBetti2}, in the case $D=0$, $\R X$ is maximal if and only if $\beta^2 = 0$ for all $\beta \in H^1(X^\circ_{\trop}, \mathcal F^1 \otimes \FF_2)$. Otherwise, if $D \neq 0$, $\R X$ is never maximal and the maximal possible first Betti number is $h^{1,1}(X) + h^{2,1}(X) -1$, which occurs if $\beta^2 + \beta \cup D = 0$ for all $\beta \in H^1(X^\circ_{\trop}, \mathcal F^1 \otimes \FF_2)$. We conjecture the following

\begin{con}
If a central primitive triangulation \(T\) admits a class 
\(\beta \in H^{1}(X^\circ_{\trop}, \mathcal F^{1} \otimes \mathbb{F}_2)\) with \(\beta^{2} \neq 0\), then no central primitive patchworking of \(\Delta\) can produce a maximal variety.
\end{con}

We believe this conjecture is reasonable following the work of Arguz-Prince \cite{AP2020}. In the context of anti-symplectic involutions preserving a Lagrangian fibration, they study the zero divisor case, and they show that the topology of the real locus does not change under ``flips'' of the triangulation (see Appendix A of \cite{AP2020}). In particular the first Betti number is independent of the triangulation. They also compute that for the mirror of the quintic in $\PP^4$, the map $\beta \mapsto \beta^2$ is not zero for a particular triangulation \cite{AP2023}, hence it is never zero. We expect that also in our context of patchworking on central triangulations, there does not exist a maximal quintic. On the other hand, in Theorem \ref{dim3_maximal}, we show that for a certain class of smooth reflexive polytopes, for any central triangulation, we can find a patchworking with $b_1(\R X) = h^{1,1}(X) + h^{1,2}(X) -1$. This applies in particular to quintics. In the context of Lagrangian fibrations, this result was also found for a particular triangulation, by the first author in \cite{Matessi2024}.

\subsection{Local configurations} \label{sec:localconf}
In this section we assume $n \geq 3$. We look for a divisor $L \in H^1(X_{\trop}^\circ, \mathcal F^1\otimes \FF_2) $ such that
\begin{equation} \label{div_maximal}
 \beta^2 + \beta \ \cup L = 0, \quad \text{for all} \ \ \beta \in H^1(X_{\trop}^\circ, \mathcal F^1\otimes \FF_2).
\end{equation}
First of all, thanks to the surjectivity Theorem \ref{thm:surj}, we may assume $\beta$ and $L$ to be the restriction of a divisor in $H^{1}(\T\Sigma_{T},\mathcal{F}^{1}\otimes \FF_2)$. Recall also that toric divisors of type $D_v$, with $v$ a vertex of $T$, restrict to zero whenever $v$ is contained in the interior of a facet of $\Delta$ (see Remark \ref{rem:divinfacet}). Therefore we may write $L$ as
\[ L = \sum_{v \in \Sk_{n-1}(\partial \Delta)}\epsilon_v D_v.\]
for some coefficients $\epsilon_v \in \FF_2$. Obviously we may also think of $L$ as the set of $v$'s such that $\epsilon_v = 1$. Then $L$ satisfies \eqref{div_maximal} if and only if 
\begin{equation} \label{divisor_max_loc}
    D_r^2 = D_r \cup L \quad \text{for all} \ r \in \Sk_{n-1} \partial \Delta. 
\end{equation} 
We now observe that $D_r \cup D_v = 0$ whenever $r$ and $v$ do not lie on a common edge of $T$, since this product already vanishes in the ambient toric variety. We say that two vertices $r$ and $v$ are adjacent, and we denote it by $r \leftrightarrow v$, if $r$ and $v$ belong to a common edge of $T$. Notice that $r \leftrightarrow r.$ Now, for every $r \in \Sk_{n-1} \partial \Delta$, we denote by 
\[ \mathcal L_r = \{v \in L \, | \, v \leftrightarrow r \  \}.\]
Then \eqref{divisor_max_loc} holds if and only if
\begin{equation} \label{dmax_loc}
   D_r\cup (D_r + \sum_{v \in \mathcal L_r} D_v) = 0  
\end{equation} 
for all $r \in \Sk_{n-1} \partial \Delta$. 

Recall that two toric divisors $D_1$ and $D_2$ are linearly equivalent (over $\FF_2$), if and only if there exists $m \in N \otimes \FF_2$ such that 
\[ D_1+D_2 = \sum_{j} \langle m, v_j \rangle D_{v_j}, \]
where the sum is over all primitive generators $v_j$ of the rays of the fan. We write, in this case, $D_1 \sim D_2$.
We now fix $r \in \Sk_{n-1} \partial \Delta$ and look for configurations of $L$ around $r$ which ensure that \eqref{dmax_loc} holds for that particular $r$.  It is convenient to denote by $o(D_r)$ any divisor which does not contain vertices adjacent to $r$. In particular $D_r \cup o(D_r) = 0$. We have that \eqref{dmax_loc} certainly holds if 
\begin{equation} \label{loc_triv}
     D_r + \sum_{v \in \mathcal L_r} D_v \, \sim \, o(D_r).
\end{equation}
We have the following cases. 

\medskip

{\bf Case 1.} If $r \in \mathcal L_r$, then \eqref{loc_triv} holds if $\sum_{v \in \mathcal L_r- \{r\}} D_v \, \sim \, o(D_r)$. This means that there exists an $m \in N \otimes \FF_2$ such that 
\begin{itemize}
        \item $\inn{m}{r} =0 \mod 2$;
        \item $\inn{m}{v} =1 \mod 2$ for all $v \in \mathcal L_r - \{r\}$;
        \item $\inn{m}{w} =0 \mod 2$ for all $w \leftrightarrow r$ such that $w \notin \mathcal L_r$.
\end{itemize}
 Notice this includes also the case where $\mathcal L_r = \{r \}$, i.e. $r$ is isolated in $L$. Indeed, in this case, we can take $m=0$.   

\medskip

{\bf Case 2.} If $r \notin \mathcal L_r$,  then \eqref{loc_triv} holds if there exists $m \in N \otimes \FF_2$ such that
\begin{itemize}
         \item $\inn{m}{r} =1 \mod 2$;
        \item $\inn{m}{v} =1 \mod 2$ for all $v \in \mathcal L_r$;
        \item $\inn{m}{w} =0 \mod 2$ for all $w \leftrightarrow r$ such that $w \notin \mathcal L_r$ and $w \neq r$.
\end{itemize}
This also includes the case where $r$ is in the interior of a codimension $2$ face of $\Delta$ and every vertex adjacent to $r$ is in $L$, except $r$. 

We have that if for every $r \in \Sk_{n-1} (\partial \Delta)$, the set $\mathcal L_r$ falls in one of the two cases above, then $L$ satisfies \eqref{div_maximal}.

\subsection{Local examples in dimension 3} \label{sec:localexample} We now assume that $\Delta \subset M_{\R}$ is a non-singular $4$-dimensional reflexive polytope with some primitive central triangulation $T$. We look at some examples of local configurations for a divisor $L$ satisfying \eqref{div_maximal}.  Define the parity of a vector $v \in M$ to be its reduction in $M \otimes \FF_2$, which we denote by $[v]$. Given a face $F$ of $\Delta$, we will denote by $[F] \subseteq M \otimes \FF_2$ the $\mod 2$ reduction of the affine subspace spanned by $F$. Notice that if $F$ is a proper face, since $\Delta$ is reflexive, $[F]$ never passes through the origin. 

\subsubsection{Interior of a $2$-dimensional face} \label{interior_face}
    Let $F$ be a $2$-dimensional face of $\Delta$.  Fix a parity $p \in [F]$. Consider a divisor $L$ such that $L \cap F$ consists only of those points of $F$ with parity $p$. Let $r$ be a point in the interior of $F$. We check if $\mathcal L_r$ falls into one of the two cases above.  Suppose first that $r \in L$. Then it is isolated. Indeed, since the triangulation $T$ is primitive, $r$ is the only point adjacent to $r$ of parity $p$. So $\mathcal L_r$ falls in Case 1. Suppose now that $r \notin L$. In particular, two of the four parities inside $[F]$ are $[r]$ and $p$. Consider a $3$-dimensional vector subspace in $M \otimes \FF_2$, transverse to $[F]$ and intersecting it precisely in the two parities which are different from $[r]$ and $p$. Let $m \in N \otimes \FF_2$ be the orthogonal vector to this three dimensional vector subspace. Then $\inn{m}{p} = \inn{m}{[r]} = 1$. If $v$ is adjacent to $r$, but $v \neq r$, then $[v] \neq [r]$. In particular,  if $v \in \mathcal L_r$ then its parity is $p$, in which case $\inn{m}{v} = 1$. Otherwise, if $v \notin \mathcal L_r$, then $[v] \neq p$ and therefore $\inn{m}{v} = 0$.  Then $\mathcal L_r$ falls in Case $2$. Notice that these arguments are independent of the triangulation $T$.
\subsubsection{Around a vertex}
Let $r$ be a vertex of $\Delta$. Assume that we have chosen a parity $p_F \in [F]$ for each two dimensional face $F$ containing $r$ and define $L$ so that $L \cap F$ consists only of points of parity $p_F$, for every $F$. Clearly, given $F$ and $F'$, if $L \cap (F \cap F') \neq \emptyset$, then we must have $p_F = p_{F'}$. Since $\Delta$ is smooth, the dual face $r^{\vee}$ is a facet of $\Delta^{\circ}$ whose vertices $n_1, \ldots, n_4$ form a basis of $N$. Moreover,  $\inn{[n_j]}{[r]} = 1$ for all $j=1, \ldots r$. There are four edges containing $r$, denote them by $E_1, \ldots, E_4$ so that $E_j$ is contained in the affine line defined by the linear equations $\inn{n_k}{v}=1$, for all $k\neq j$. There are six $2$-dimensional faces containing $r$. Denote by $F_{jk}$ the one that contains $E_j$ and $E_k$, ($j \neq k$). We have that $F_{jk}$ is contained in the affine plane given by equations $\inn{n_i}{v}=1$ for all $i \neq j, k$. Denote by $v_j$ the integral point of $E_j$, adjacent to $r$ (and different from $r$). Notice that 
\begin{equation} \label{n_on_v}
    \inn{[n_j]}{[v_j]} = 0 \quad \text{and} \quad  \inn{[n_i]}{[v_j]} = 1 \quad \text{for all} \ i \neq j
\end{equation}
Let $p_{jk}$ denote the chosen parity on the face $F_{jk}$. Among all possible configurations, we look for the admissible ones, i.e. those where  $\mathcal L_r$ satisfies Case 1 or Case 2. 

\begin{ex}[The simple vertex] \label{simple_vertex}
  Suppose that $p_{jk} = [r]$ for all $j$ and $k$. Then $r \in \mathcal L_r$ and it is isolated. Therefore the configuration is admissible since it satisfies Case 1.  See Figure \ref{figure1}. 
\end{ex}
\begin{figure}[ht]
  \begin{minipage}[b]{0.45\linewidth}

    \centering
    \begin{tikzpicture}
        \draw (0,0) node[below] {$r$}--(2,0) node[right] {$E_1$} ;
        \fill[red] (0,0) circle (1.5pt);
        \fill[black] (1,0) circle (1.5pt);
        \fill[red] (2,0) circle (1.5pt);
        \draw (0,0)--(1.5,1.5) node[above right] {$E_2$} ;
        \fill[black] (0.75,0.75) circle (1.5pt);
        \fill[red] (1.5,1.5) circle (1.5pt);
         \draw (0,0)--(0,2) node[above] {$E_3$};
        \fill[black] (0,1) circle (1.5pt);
        \fill[red] (0,2) circle (1.5pt);
        \draw (0,0)--(1.5,-1.5) node[below right] {$E_4$};
        \fill[black] (0.75,-0.75) circle (1.5pt);
        \fill[red] (1.5,-1.5) circle (1.5pt);
    \end{tikzpicture}
    \end{minipage} \hfill
    \begin{minipage}[b]{0.45\linewidth}
    \begin{tikzpicture}
        \draw (0,1) node[below] {$r$}--(2,1) node[right] {$E_j$} ;
        \fill[red] (0,1) circle (1.5pt);
        \fill[black] (1,1) circle (1.5pt);
        \fill[red] (2,1) circle (1.5pt);
        \draw (0,1)--(0,3) node[above] {$E_k$};
        \fill[black] (0,2) circle (1.5pt);
        \fill[red] (0,3) circle (1.5pt);
        \fill[black] (1,2) circle (1.5pt);
        \fill[black] (2,2) circle (1.5pt);
        \fill[black] (1,3) circle (1.5pt);
        \fill[red] (2,3) circle (1.5pt);
    \end{tikzpicture}
    \end{minipage}
    \caption{The simple vertex: the chosen parity (in red) in any face is the parity of the vertex.}
    \label{figure1}
\end{figure}
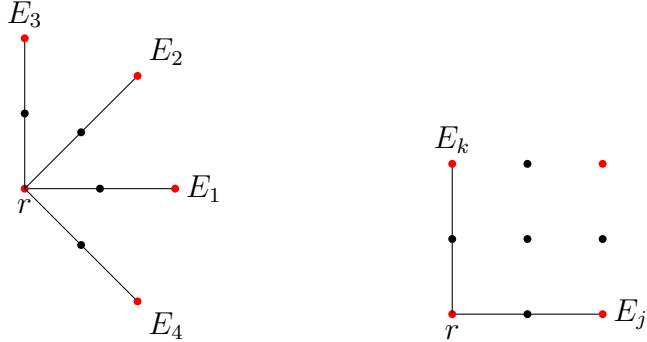

\begin{ex}[The pointed vertex] \label{pointed_vertex}
Fix some $j$ and define the configuration as follows. On all faces containing $E_j$ the chosen parity coincides with $[v_j]$, i.e. $p_{jk} = [v_j]$ for all $k$. On the remaining faces $F_{ik}$ (with $j \neq i,k$), the parity $p_{ik}$ is the unique one which is different from $[r]$, $[v_i]$ and $[v_k]$.  In this case $r \notin \mathcal L_r$. Let us check that $\mathcal L_r$ satisfies Case 2. Consider $m = \sum_{i \neq j} [n_i]$. First of all $\inn{m}{[r]}= 1$. By \eqref{n_on_v}, $\inn{m}{p_{jk}}= \inn{m}{[v_j]}= 1$. Now consider $p_{ik}$ with $j \neq i,k$. We have that $\inn{[n_\ell]}{p_{ik}}=1$ for all $\ell \neq i, k$ while $\inn{[n_i]}{p_{ik}} = \inn{[n_k]}{p_{ik}}= 0$. In particular $\inn{m}{p_{ik}} = 1$.
This shows that $m$ satisfies the first two bullets listed in Case 2. Suppose now that $w$ is adjacent to $r$, $w \notin \mathcal L_r$ and $w \neq r$. If $w \in F_{ik}$ with $j \neq i,k$, then its parity coincides with either $v_i$ or $v_k$. But then, by \eqref{n_on_v}, $\inn{m}{[w]} = 0$. On the other hand, if $w \in F_{jk}$, then either $[w] = [v_k]$, or $[w] \neq [v_j], [v_k], [r]$. In the first case we already know that $\inn{m}{[w]} = 0$, for the previous argument. In the second case,  $\inn{[n_\ell]}{[w]}=1$ for all $\ell \neq j, k$ and $\inn{[n_j]}{[w]} = \inn{[n_k]}{[w]}= 0$. Then $\inn{m}{[w]} = 0$. Therefore the configuration is admissible since $\mathcal L_r$ satisfies Case 2. See Figure \ref{figure2}.
\end{ex}

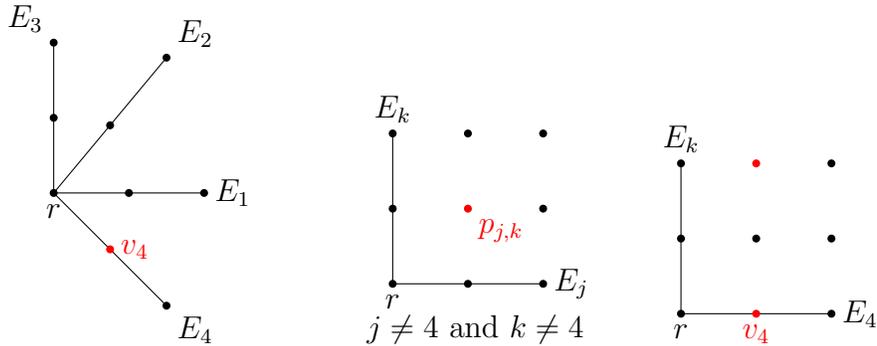
\begin{figure}[ht]
\begin{minipage}[b]{0.3\linewidth}
    \centering
    \begin{tikzpicture}
        \draw (0,0) node[below] {$r$}--(2,0) node[right] {$E_1$} ;
        \fill[black] (0,0) circle (1.5pt);
        \fill[black] (1,0) circle (1.5pt);
        \fill[black] (2,0) circle (1.5pt);
        \draw (0,0)--(1.5,1.8) node[above right] {$E_2$} ;
        \fill[black] (0.75,0.9) circle (1.5pt);
        \fill[black] (1.5,1.8) circle (1.5pt);
         \draw (0,0)--(0,2) node[above left] {$E_3$};
        \fill[black] (0,1) circle (1.5pt);
        \fill[black] (0,2) circle (1.5pt);
        \draw (0,0)--(1.5,-1.5) node[below right] {$E_4$};
        \fill[red] (0.75,-0.75) node[right] {$v_{4}$} circle (1.5pt);
        \fill[black] (1.5,-1.5) circle (1.5pt);
    \end{tikzpicture}
    \end{minipage}
    \hfill
    \begin{minipage}[b]{0.3\linewidth}
    \begin{tikzpicture}
        \draw (0,0) node[below] {$r$}--(2,0) node[right] {$E_j$} ;
        \fill[black] (0,0) circle (1.5pt);
        \fill[black] (1,0)   circle (1.5pt);
        \draw node at (1.1,-0.6) {$j\neq 4$ and $k\neq 4$};
        \fill[black] (2,0) circle (1.5pt);
        \draw (0,0)--(0,2) node[above] {$E_k$};
        \fill[black] (0,1) circle (1.5pt);
        \fill[black] (0,2) circle (1.5pt);
        \fill[black] (2,1) circle (1.5pt);
        \fill[red] (1,1) node[below right] {$p_{j,k}$} circle (1.5pt);
        \fill[black] (2,2) circle (1.5pt);
        \fill[black] (1,2) circle (1.5pt);
    \end{tikzpicture}
    \end{minipage}
    \begin{minipage}[b]{0.3\linewidth}
            \begin{tikzpicture}
        \draw (0,0) node[below] {$r$}--(2,0) node[right] {$E_4$} ;
        \fill[black] (0,0) circle (1.5pt);
        \fill[red] (1,0)  node[below] {$v_4$} circle (1.5pt);
        \fill[black] (2,0) circle (1.5pt);
        \draw (0,0)--(0,2) node[above] {$E_k$};
        \fill[black] (0,1) circle (1.5pt);
        \fill[black] (0,2) circle (1.5pt);
        \fill[black] (2,1) circle (1.5pt);
        \fill[black] (1,1)  circle (1.5pt);
        \fill[black] (2,2) circle (1.5pt);
        \fill[red] (1,2) circle (1.5pt);
    \end{tikzpicture}
    \end{minipage}
    \caption{The pointed vertex for $E_4$. The chosen parity is drawn in red in any face.}
    \label{figure2}
\end{figure}

\begin{ex}[Non-admissible vertex]  
Suppose that for all $j$ and $k$, $p_{jk}$ is different from the parities $[r]$, $[v_j]$ and $[v_k]$. So that $r$ and the $v_j$'s are not in $\mathcal L_r$. Since the parities $[v_1], \ldots, [v_4]$ form a basis of $M \otimes \FF_2$, the only $m \in N \otimes \FF_2$ which satisfies $\inn{m}{[v_j]} =0$ is $m=0$. But then we also have $\inn{m}{[r]} =0$. Therefore this configuration is not admissible since it cannot satisfy neither Case 1 nor Case 2. 
\end{ex}
\subsubsection{Around an edge}
Let $E$ be an edge of $\Delta$. Again, assume that we have chosen a parity $p_F \in [F]$ for each two dimensional face $F$ containing $E$ and define $L$ so that, $L \cap F$ consists only of points of parity $p_F$, for every $F$. 
The dual face $E^{\vee}$ in $\Delta^{\circ}$ is a two dimensional simplex with vertices $n_1, n_2, n_3$, so that $E$ is contained in the affine line given by equations $ \inn{n_j}{v}=1$ for $j = 1,2,3$. There are three $2$-dimensional faces containing $E$. Denote by $F_j$ the one defined by the equations $\inn{n_k}{v}=1$ for all $k \neq j$. We look for admissible configurations around an edge $E$, i.e. those for which $\mathcal L_r$ satisfies Case 1 or Case 2 for every $r$ in the interior of $E$.

\begin{ex}[The simple edge] \label{simple_edge}
Choose a parity $p_E$ on $[E]$. Define $p_F = p_E$ for every $F$ containing $E$. If $[r]=p_E$, then $r$ is isolated and therefore $\mathcal L_r$ satisfies Case 1. If $[r]\neq p_E$, let $m= \sum_{j=1}^{3} [n_j]$. Then $\inn{m}{[r]}=\inn{m}{p_E}= 1$. So that $\mathcal L_r$ satisfies the first two bullets of Case 2. The affine hyperplane in $M \otimes \FF_2$ cut by the equation $\inn{m}{v}=1$, is transverse to all affine planes $[F_j]$ and intersects them exactly in the two points $[r]$ and $p_E$. In particular, for any $w \in F$ adjacent to $r$, $w \notin \mathcal L_r$ and $w \neq r$, we must have $\inn{m}{w}=0$. Therefore $\mathcal L_r$ satisfies Case 2 and the configuration is admissible. See Figure \ref{figure3}. Notice that any edge $E$ adjacent to the simple vertex in Example \ref{simple_vertex} replicates the configuration just described here if $p_E$ is the parity of the vertex. Similarly, in the pointed vertex Example \ref{pointed_vertex}, this configuration is replicated by the edge $E_j$, when $p_E$ is the parity of $v_j$. 
\end{ex}
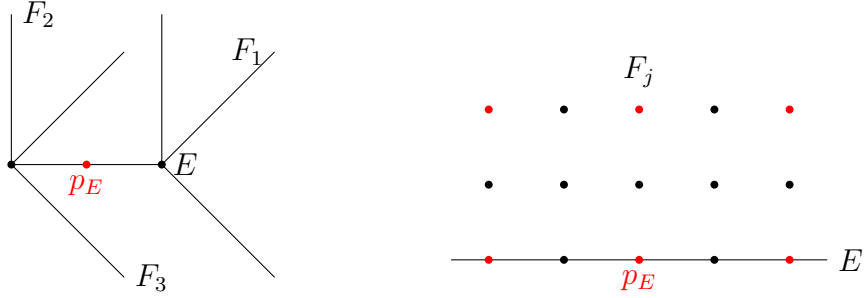
\begin{figure}[ht]
  \begin{minipage}[b]{0.45\linewidth}

    \centering
    \begin{tikzpicture}
        \draw (0,0)--(2,0) node[right] {$E$} ;
        \fill[black] (0,0) circle (1.5pt);
        \fill[red] (1,0) node[below] {$p_E$} circle (1.5pt);
        \fill[black] (2,0) circle (1.5pt);
        \draw (0,0)--(1.5,1.5) ;
        \draw (2,0)--(3.5,1.5) node[left] {$F_1$};
        \draw (0,0)--(0,2) node[right] {$F_2$};
        \draw (2,0)--(2,2) ;
        \draw (0,0)--(1.5,-1.5) node[right] {$F_3$};
        \draw (2,0)--(3.5,-1.5) ;
    \end{tikzpicture}
    \end{minipage} \hfill
    \begin{minipage}[b]{0.45\linewidth}
    \begin{tikzpicture}
        \draw (-1.5,1)--(3.5,1) node[right] {$E$} ;
        \fill[black] (0,1) circle (1.5pt);
        \fill[red] (1,1) node[below] {$p_E$} circle (1.5pt);
        \fill[black] (2,1) circle (1.5pt);
        \fill[black] (0,2) circle (1.5pt);
        \fill[black] (0,3) circle (1.5pt);
        \fill[black] (1,2) circle (1.5pt);
        \fill[black] (2,2) circle (1.5pt);
        \fill[red] (1,3) circle (1.5pt);
        \fill[red] (3,1) circle (1.5pt);
        \fill[red] (-1,1) circle (1.5pt);
        \fill[red] (3,3) circle (1.5pt);
        \fill[red] (-1,3) circle (1.5pt);
        \fill[black] (3,2) circle (1.5pt);
        \fill[black] (-1,2) circle (1.5pt);
        \fill[black] (2,3)  circle (1.5pt);
        \draw node at (1,3.5) {$F_j$};
    \end{tikzpicture}
    \end{minipage}
    \caption{The simple edge: the chosen parity (in red) in any face is a parity of the edge.}
    \label{figure3}
\end{figure}
\begin{ex}[The empty edge]
   In this configuration none of the chosen parities $p_F$ lie on $E$, in particular $E$ does not contain elements of $L$. First choose $p(F_1)$ and $p(F_2)$ (not lying on $E$). Then there is just one parity on $F_3$, not lying on $E$, which is linearly independent of $p(F_1)$ and $p(F_2)$. Let $p(F_3)$ be this parity. For any $r$ in $E$, the points $[r], p_{F_1}, p_{F_2}, p_{F_3}$ form a basis of $M \otimes \FF_2$. In particular there exists a unique affine hyperplane passing through  $[r], p_{F_1}, p_{F_2}, p_{F_3}$ (and not through the origin), i.e. there is an $m$ such that $\inn{m}{[r]} = \inn{m}{p_{F_j}}= 1$ for all $j=1,2,3$. Such an hyperplane must also be transverse to all $[F_j]$'s. In particular if $w \in F$ is adjacent to $r$, $w \notin \mathcal L_r$ and $w \neq r$, we must have $\inn{m}{[w]} = 0$. Therefore this configuration is admissible, since $\mathcal L_r$ satisfies Case 2 for all $r$ in the interior of $E$. See Figure \ref{figure4}. Notice that in the pointed vertex Example \ref{pointed_vertex}, this configuration is replicated by every edge $E_k$ with $k \neq j$.  
\end{ex}
\begin{figure}[ht]
  \begin{minipage}[b]{0.45\linewidth}

    \centering
    \begin{tikzpicture}
        \draw (0,0)--(2,0) node[right] {$E$} ;
        \fill[black] (0,0) circle (1.5pt);
        \fill[black] (1,0) circle (1.5pt);
        \fill[black] (2,0) circle (1.5pt);
        \draw (0,0)--(1.5,1.5) ;
        \draw (2,0)--(3.5,1.5) node[left] {$F_1$};
        \draw (0,0)--(0,2) node[right] {$F_2$};
        \draw (2,0)--(2,2) ;
        \draw (0,0)--(1.5,-1.5) node[right] {$F_3$};
        \draw (2,0)--(3.5,-1.5) ;
    \end{tikzpicture}
    \end{minipage} \hfill
    \begin{minipage}[b]{0.45\linewidth}
    \begin{tikzpicture}
        \draw (-1.5,1)--(3.5,1) node[right] {$E$} ;
        \fill[black] (0,1) circle (1.5pt);
        \fill[black] (1,1)  circle (1.5pt);
        \fill[black] (2,1) circle (1.5pt);
        \fill[black] (0,2) circle (1.5pt);
        \fill[black] (0,3) circle (1.5pt);
        \fill[red] (1,2) node[below] {$p(F_j)$} circle (1.5pt);
        \fill[black] (2,2) circle (1.5pt);
        \fill[black] (1,3) circle (1.5pt);
        \fill[black] (3,1) circle (1.5pt);
        \fill[black] (-1,1) circle (1.5pt);
        \fill[black] (3,3) circle (1.5pt);
        \fill[black] (-1,3) circle (1.5pt);
        \fill[red] (3,2) circle (1.5pt);
        \fill[red] (-1,2) circle (1.5pt);
        \fill[black] (2,3)  circle (1.5pt);
        \draw node at (1,3.5) {$F_j$};
    \end{tikzpicture}
    \end{minipage}
    \caption{The empty edge: the chosen parity is drawn in red in any face.}
    \label{figure4}
\end{figure}
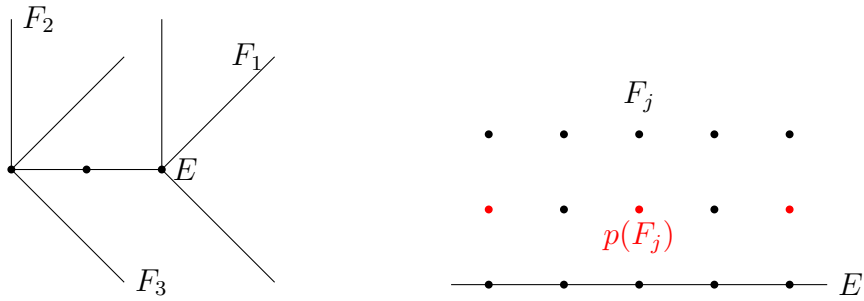
\begin{ex}[The non-admissible edge]
 Having chosen $p_{F_1}$ and $p_{F_2}$, not lying on $E$, one can choose $p_{F_3}$, not lying on $E$, so that $p_{F_1}+ p_{F_2}+ p_{F_3} = 0$. In particular, there does not exist $m$ such that $\inn{m}{p_{F_j}}= 1$ for all $j=1,2,3$. Hence this configuration is not admissible since $\mathcal L_r$ cannot satisfy Case 2. 
\end{ex}

\subsection{Global examples in dimension 3} \label{sec:globalex}
We describe a simple way to construct a divisor $L$ satisfying  \eqref{div_maximal}, which is independent of the chosen primitive triangulation $T$. The method works well on some examples of reflexive polytopes.

\begin{thm} \label{dim3_maximal} Let $\Delta \subset M_{\R}$ be a smooth reflexive polytope of dimension $4$ satisfying one of the following hypothesis
\begin{itemize}
    \item for every $2$-dimensional face $F$ of $\Delta$, the boundary $\partial F$ contains only three parities;
    \item every edge of $\Delta$ has even length.
\end{itemize}
Then, for every primitive central triangulation $T$ of $\Delta$, there exists a divisor $L \in H^1(X_{\trop}^\circ, \mathcal F^1 \otimes \FF_2)$ satisfying \eqref{div_maximal}.   
\end{thm}

\begin{proof} 
The easiest case is when $\Delta$ satisfies the second bullet. With this hypothesis, all vertices $v$ of $\Delta$ have the same parity, i.e. there exists $p \in M \otimes \FF_2$ such that $[v]= p$ for every vertex $v \in \Delta$. Let $L$ consist of all the points $r \in \Sk_2(\Delta) \cap M$ of parity $p$. We have $p \in [F]$ for all proper faces of $\Delta$, therefore, in the interior of every two dimensional face, $L$ is like in \S \ref{interior_face}. Every vertex is a simple vertex (Example \ref{simple_vertex}) and every edge is a simple edge (Example \ref{simple_edge}).

Assume now that $\Delta$ satisfies the first bullet. Choose a vertex $v_0 \in \Delta$ and let $p=[v_0]$. Let $L$ be the divisor such that, for every $2$-dimensional face $F$, $L \cap F$ consists of points of parity $p$ if $p \in [F]$, otherwise it consists of points of the unique parity in $[F]$ which is not contained in $\partial F$. Then, in the interior of $F$, $L$ is as in Example \ref{interior_face}. Notice that $v_0$ is a simple vertex. 

Suppose now that $v \in \Delta$ is a simple vertex with respect to $L$. Then $[v] = p$ and every edge containing $v$ is a simple edge. We have that any other vertex $v'$ of $\Delta$ which is connected to $v$ by an edge $E$, is either a simple or a pointed vertex. Indeed, if $E$ has even length, then also $[v']= p$, and $v'$ is a simple vertex by definition. If $E$ has odd length, then $[v'] \neq p$. Let $v''$ be the integral point in the interior of $E$ which is closest to $v'$, then $[v''] = p$. Let $F$ be a two dimensional face containing $v'$. If $F$ contains $E$, then by definition $L \cap F$ consists of points of parity $p$. If $F$ does not contain $E$, then $p \notin [F]$, therefore $L \cap F$ consists of points whose parity is the unique one in $[F]$ which is not on $\partial F$. Hence $v'$ must be pointed vertex. 

Suppose now that $v$ is a pointed vertex and denote by $E_s$ the unique simple edge containing $v$. Let us prove that any other vertex $v'$ of $\Delta$ which is connected to $v$ by an edge $E$, is either a simple or a pointed vertex. Assume first that $E=E_s$. If $E$ has odd length, $[v'] = p$ and hence $v'$ is a simple vertex. If $E$ has even length, as we argued in the previous case, $v'$ must be a pointed vertex. Now assume $E$ is one of the empty edges. Let $F_s$ be the two dimensional face which contains $E$ and $E_s$. Obviously $v'$ is a vertex of $F_s$. Let $E'_s$ be the other edge of $F_s$ which contains $v'$ and let $v''$ be the integral point in the interior of $E'_s$ which is closest to $v'$. By the property that $\partial F_s$ contains only three parities, we must have that $[v''] = p$, since none of the two parities in $E$ can be equal to $p$. Now let $F$ be a face containing $E$, but not $E_s$. Notice that since $v$ is a pointed vertex, $L \cap F$ is the set of points in $F$ whose parity is the unique one not contained in $\partial F$. Let $E'$ be the other edge of $F$ which contains $v'$. Then $L \cap E' = \emptyset$. Let $F'$ be a two dimensional face containing $v'$ but not $E$. If it contains $E'_s$, then $L \cap F'$ is the set points of parity $p$. If $F'$ does not contain $E'_s$, then $p \notin [F']$ and $L \cap F'$ is set of points whose parity is the unique one not in $\partial F'$. This shows that $v'$ is a pointed vertex. 

Given any vertex $v$ of $\Delta$, we can connect it to $v_0$ by a path of edges, and therefore $v$ is either simple or pointed. Moreover, also every edge is either simple or empty. This concludes the proof of the theorem. 
\end{proof}

\begin{cor}
    Under the same assumptions as in Theorem~\ref{dim3_maximal}, for any primitive triangulation \(T\) of \(\Delta\) there exists a divisor \(D\) such that the real hypersurface \(\mathbb{R}X_D\) is connected and
    \[
        b_1(\mathbb{R}X_D) = h^{1,2}(X)+h^{1,1}(X)-1.
    \]
\end{cor}

\begin{ex} Any smooth reflexive $4$ dimensional polytope whose two dimensional faces are simplices, satisfies the second bullet. 
For instance let $P$ be a standard (unimodular) $4$ dimensional simplex in $\R^4$. Then $\Delta=5P$ is the reflexive polytope which corresponds to quintics in $\PP^4$. Any central triangulation of $\Delta$ admits a patchworking real quintic $\R X$ with $b_1(\R X) = h^{1,1}(X) + h^{1,2}(X)-1 = 101$. A similar construction, on a topological model of the quintic, was given by the first author in the context of antisymplectic involutions preserving Lagrangian fibrations, see Section 8.2 of \cite{Matessi2024}. The main difference here is that via patchworking we find actual quintics, instead of just a topological model. Moreover, Theorem \ref{dim3_maximal} proves existence for any central triangulation instead of a specific one.
\end{ex}

\begin{ex} Consider the cube $\square=[-1,1]^4$. This is the polytope for a Calabi-Yau hypersurface $X$ of degree $(2,2,2,2)$ in $(\C P^1)^4$. All of its edges have length $2$, therefore for any central triangulation of $\square$ we can find a patchworking $\R X$ with first Betti number $b_1(\R X) = h^{1,1}(X) + h^{1,2}(X)-1$.  
\end{ex} 

\begin{ex} Of course there are a lot of examples where Theorem \ref{dim3_maximal} does not apply. A simple example is the following. Let $P$ be a standard $2$-dimensional simplex and let $\Delta = (3P) \times (3P)$. It corresponds to hypersurfaces in $\C P^2 \times \C P^2$ of degree $(3,3)$. Any square two dimensional face $F$ of $\Delta$ is such that $\partial F$ contains $4$-parities, therefore Theorem \ref{dim3_maximal} does not apply. It is unlikely that a divisor satisfying \eqref{div_maximal} exists for all triangulations. For the moment we have been unable to find a triangulation admitting such a divisor. 
\end{ex}

\subsection{Higher dimensions} The problem of maximizing the first Betti number in higher dimensions is harder. Indeed in higher dimensions there is an interesting obstruction to finding a divisor satisfying \eqref{div_maximal}. Suppose such a divisor exists. Given $\alpha, \beta \in H^1(X_{\trop}^\circ, \mathcal F^1)$ we have
\[
\begin{split}
    \beta^2 &= L \cup \beta \\
    \alpha^2 & = L \cup \alpha.
\end{split}
\]
If we multiply the first equation by $\alpha$, the second by $\beta$ and then subtract the two equations, we get that
\[ \alpha^2 \cup \beta = \alpha \cup \beta^2 \quad \text{for all} \ \ \ \alpha, \beta \in H^1(X_{\trop}^\circ, \mathcal F^1).\]
It can be shown that in dimension $n=3$, this relation is indeed satisfied on a mirror $X_{\trop}^\circ$ associated to any central triangulation. On the contrary, in higher dimensions, not all mirrors $X_{\trop}^\circ$, i.e. not all triangulations, satisfy this property. More details on this obstruction and on the higher-dimensional problem will be given in our forthcoming paper   
\cite{MR26}.

\bibliographystyle{plain}
\bibliography{biblio}
\end{document}